\documentclass[12pt,twoside,reqno]{amsart}

\usepackage[OT1]{fontenc}
\usepackage{type1cm}

\usepackage{comment}
\usepackage{enumerate}

\usepackage{amsthm}
\RequirePackage{amsmath,amsfonts,amssymb,amsthm}

\RequirePackage[dvips]{graphicx}

\usepackage{psfrag}
\usepackage[usenames]{color}
  \numberwithin{equation}{section}

\usepackage[active]{srcltx}  

  \newcommand{\N}{\mathbb{N}}         
  \newcommand{\R}{\mathbb{R}}         
  \newcommand{\Z}{\mathbb{Z}}         

  \newcommand{\EE}{\mathbb{E}}        
  \newcommand{\PP}{\mathbb{P}}        

\newcommand{\jjj}{\mathtt{j}}

  \newcommand{\e}{\varepsilon}
  \newcommand{\wh}{\widehat}

  \newcommand{\cJ}{\mathcal{J}}
  \newcommand{\1}{\mathbf{1}}

  \newcommand{\AP}{\mathsf{AP}}
  \newcommand{\hdim}{\dim_\mathsf{H}}

  \DeclareMathOperator{\supp}{supp}   

  \newtheorem{thm}{Theorem}[section]
  \newtheorem{mainthm}{Theorem}
  
  \newtheorem{lemma}[thm]{Lemma}

  \theoremstyle{remark}

\begin{document}

\title{Salem sets with no arithmetic progressions}

\author{Pablo Shmerkin}
\address{Department of Mathematics and Statistics, Torcuato Di Tella University, and CONICET, Buenos Aires, Argentina}
\email{pshmerkin@utdt.edu}
\thanks{The author was partially supported by Projects PICT 2011-0436 and PICT 2013-1393 (ANPCyT)}
\urladdr{http://www.utdt.edu/profesores/pshmerkin}



\begin{abstract}
We construct compact Salem sets in $\R/\Z$ of any dimension (including $1$) which do not contain any arithmetic progressions of length $3$. Moreover, the sets can be taken to be Ahlfors regular if the dimension is less than $1$, and the measure witnessing the Fourier decay can be taken to be Frostman in the case of dimension $1$. This is in sharp contrast to the situation in the discrete setting (where Fourier uniformity is well known to imply existence of progressions), and helps clarify a result of {\L}aba and Pramanik on pseudo-random subsets of $\R$ which do contain progressions.
\end{abstract}

\maketitle

\section{Introduction and statement of results}

A fundamental notion of pseudo-randomness for sets $A\subset\Z/n\Z$ is Fourier uniformity in the sense that $|\widehat{\mathbf{1}}_A(k)|$ is small for all $k\neq 0$. This notion of pseudo-randomness is particularly well suited to study the presence of non-trivial $3$-term arithmetic progressions (from now on, 3-$\AP$s) inside $A$. For example, if $A\subset\Z/n\Z$ and
\begin{equation} \label{eq:uniformiti-implies-APs}
|\widehat{\mathbf{1}}_A(k)|< \frac{|A|^2}{n^2} - \frac{1}{n} \quad\text{for all } k\neq 0,
\end{equation}
then $A$ contains a 3-$\AP$, and similar considerations are crucial in Roth's proof of his famous theorem on $3$-$\AP$s inside sets of positive density. See \cite[Chapter 10]{TaoVu10} for details. It can be checked that \eqref{eq:uniformiti-implies-APs} holds for random sparse sets of size as small as $\delta n^{2/3}$.

For subsets of the circle $\R/\Z$, an analog notion of pseudo-randomness is Fourier decay rather than Fourier uniformity, and it is therefore natural to investigate whether a connection with $3$-$\AP$s still persists. Since a set of positive Lebesgue measure contains $\AP$s of any length (as a consequence of the Lebesgue density lemma), one is mostly interested in sparse sets, particularly sets of fractional Hausdorff dimension. It is well known that the speed of Fourier decay of a measure is limited by the ``size'' of the measure (by a measure we always mean a Borel probability measure). For example, if $\mu$ is a measure on $\R/\Z$ and $\mu(A)>0$ for some Borel set $A$ with $\hdim(A)<t<1$ (we denote Hausdorff dimension by $\hdim$), then
\[
\sum_{k\in\Z\setminus\{0\}} |\wh{\mu}(k)|^2 |k|^{t-1} = +\infty.
\]
In particular, if $|\wh{\mu}(k)|\le C |k|^{-\sigma/2}$, then $\sigma\le t$. This leads to the notion of \emph{Salem set}: a Borel set $A\subset\R/\Z$ is a Salem set if for every $t\in (0,\hdim A)$ there exists a measure $\mu$ with $\supp(\mu)\subset A$ and such that $|\wh{\mu}(k)|\le C_t |k|^{t/2}$. See \cite[Chapter 3]{Mattila15} for background on the Fourier transform of measures, its relationship with dimension, and further discussion on Salem and related sets.

Salem sets, and more generally measures with large Fourier decay, indeed exhibit pseudo-random behavior in some arithmetic aspects. For example, if $A$ is Salem, then
\[
\hdim(A+B) = \min(\hdim(A)+\hdim(B),1)
\]
for any Borel set $B\subset\R/\Z$. This follows from Frostman's Lemma (applied to $B$) together with the Fourier transform expression for the energy, see \cite[Theorem 3.10]{Mattila15}.  As another example, Mockenhaupt \cite{Mockenhaupt00} (other than the endpoint) and Bak and Seeger \cite{BakSeeger11} (at the endpoint) showed that measures with Fourier decay that, additionally, have a suitable uniform mass decay, satisfy a restriction theorem analog to the Stein-Tomas restriction theorem for the sphere (in fact, the Stein-Tomas Theorem is a special case of the Mockenhaupt-Bak and Seeger results - note that the sphere $S^{d-1}$ is a Salem set in $\R^d$).

So what is the connection between Fourier decay and $3$-$\AP$s for subsets of the circle or the real line?

A partial answer was provided by I. {\L}aba and M. Pramanik in \cite[Theorem 1.2]{LabaPramanik09}. For clarity, we do not state their result in full generality.
\begin{thm}[{\L}aba and Pramanik] \label{thm:LabaPramanik}
Given $C_1, C_2>0$ there exists $\e_0=\e_0(C_1,C_2)>0$ such that the following holds: if  $\beta>4/5, \alpha>1-\e_0$, and $\mu$ is a measure on $[0,1]$ such that
\begin{enumerate}
\item $\mu(x,x+r) \le C_1 r^\alpha$ for all $x\in \R/\Z$ and all $r\in (0,1)$,
\item $|\wh{\mu}(k)| \le C_2 |k|^{-\beta/2}$ for all $k\neq 0$,
\end{enumerate}
then $\supp(\mu)$ contains a $3$-$\AP$.
\end{thm}
See also \cite{CLP14} for a generalization to higher dimensions and  more general patterns. While, informally, the above result can be interpreted as saying that large Fourier decay together with large power mass decay imply the existence of progressions, the actual statement is much more subtle because of the dependence of the required decay on the constants involved. Hence, this result left open the question of whether large Fourier decay alone, or at least in combination with large mass decay (that is, regardless of the constants $C_1, C_2$), is enough to force the existence of progressions. The purpose of this article is to give an emphatic negative answer to this question.

\begin{mainthm} \label{mainthm:dimless1}
For every $t\in (0,1)$, there exists a Borel probability measure $\mu$ on $\R/\Z$ (or on $[0,1]$) such that:
\begin{enumerate}
\item[\textsf{(i)}] For each $\sigma<t/2$ there exists $C_\sigma>0$ such that
\[
|\wh{\mu}(k)| \le C_\sigma |k|^{-\sigma} \quad\text{for all }k\in\Z\setminus \{0\}.
\]
\item[\textsf{(ii)}] There exists a constant $C>0$ such that
\begin{align*}
\mu(x-r,x+r)&\le C r^t \quad\text{for all } x\in\R/\Z, r\in (0,1]\\
\mu(x-r,x+r)&\ge C^{-1} r^t \quad\text{for all } x\in\supp(\mu), r\in (0,1].
\end{align*}
\item[\textsf{(iii)}] The topological support of $\mu$ contains no non-trivial arithmetic progressions of length $3$.
\end{enumerate}
\end{mainthm}

Note, in particular, that $\supp(\mu)$ is a Salem set of Hausdorff dimension $t$. Much more is true: the second condition says that $\mu$ is a $t$-Ahlfors-David regular measure. It will be clear from the proof that the constant $C$ in the second part blows up as $t\uparrow 1$, in consonance with Theorem \ref{thm:LabaPramanik}.

Clearly, this result cannot hold for $t=1$, since the second part would then imply that $\mu$ is absolutely continuous, whence its support would contain plenty of long progressions. Nevertheless, a slight weakening still holds:
\begin{mainthm} \label{mainthm:dimeq1}
There exists a Borel probability measure $\mu$ on $\R/\Z$ such that for each $t\in (0,1)$ there are constants $C_t, C'_t>0$ such that
\begin{enumerate}
\item[\textsf{(i)}] $|\wh{\mu}(k)| \le C_t |k|^{-t/2}$ for all $k\in\Z\setminus \{0\}$.
\item[\textsf{(ii)}] $\mu(x-r,x+r) \le C'_t r^{t}$ for all $x\in\R$, $r\in (0,1]$.
\item[\textsf{(iii)}] The topological support of $\mu$ contains no non-trivial arithmetic progressions of length $3$.
\end{enumerate}
\end{mainthm}

In particular, $\supp(\mu)$ is a $1$-dimensional Salem set without progressions. We note that Keleti \cite{Keleti98, Keleti08} constructed compact sets $A\subset\R$ of Hausdorff dimension $1$ which do not contain $3$-$\AP$s (and indeed more general patterns such as rectangles $x, x+r, y, y+r$). It is unclear if his constructions are (or can be modified to be) Salem sets. Keleti's constructions are deterministic and it is notoriously difficult to establish that given deterministic sets support measures with any polynomial Fourier decay, let alone maximal decay.

Instead, we use Behrend's classical examples of large sets in $\Z/n\Z$ without progressions, together with a variant of a result from \cite{ShmerkinSuomala15}, asserting that a wide class of random measures has maximal Fourier decay. The randomization will consist in randomly translating (in a suitable sense) the Behrend-like sets in question. That is, we exploit the trivial observation that if a set in $\Z/n\Z$ has no progressions, then neither does any translation. On the other hand, these translations introduce enough uniformity in the measures to guarantee fast Fourier decay.

To conclude this introduction, we point out that the dual question of whether Salem sets can have \emph{more} additive structure than what one expects from their size was addressed by Hambrook and {\L}aba in \cite{HambrookLaba13}, as part of their proof of the sharpness of the restriction theorem of Mockenhaupt mentioned above. In particular, they construct Salem sets which, at many scales, contain very long progressions.

\bigskip

\textbf{Acknowledgment}. I thank I. {\L}aba for useful discussions.

\section{Proofs}

\subsection{A class of Salem measures}
\label{subsec:Salem}

In this section, we introduce a class of random measures and prove that a.s. they have large Fourier decay. Specific instances of this construction will later be used to prove Theorems \ref{mainthm:dimless1} and \ref{mainthm:dimeq1}. The construction is closely related to that in \cite[Section 6]{LabaPramanik09} even though, in that paper, the goal was to exhibit sets \emph{with} $3$-$\AP$s. It is also inspired by the large class of spatially independent martingales introduced in \cite{ShmerkinSuomala15}. Indeed, the measures that witness the validity of Theorem \ref{mainthm:dimless1} are spatially independent martingales, while the measures we use in Theorem \ref{mainthm:dimeq1} narrowly fail to be, but still enjoy many similar properties.

We fix sequences of integers $(L_n)_{n=1}^\infty$, $(M_n)_{n=1}^\infty$, with $1\le L_n \le M_n$ and $2\le M_n$ for all $n$. Later we will impose further conditions on these sequences; for now, we introduce some notation. Write $[M]=\{0,1,\ldots, M-1\}$, and define
\begin{align*}
 \Sigma_n &= \{ \jjj=(j_1\ldots j_n): j_i\in [M_i]\},\\
 \Sigma^* &= \cup_{n=0}^\infty \Sigma_n.
\end{align*}
We note that $\Sigma_0$ consists of the empty word $\varnothing$. To each $\jjj=(j_1\ldots j_n)\in\Sigma_n$ we associate the interval
\[
I_\jjj = \left[\sum_{i=1}^{n}  \frac{j_{i}}{M_1\cdots M_i}, \frac{1}{M_1\ldots M_n}+\sum_{i=1}^{n} \frac{j_{i} }{M_1\cdots M_i} \right).
\]
These are the basic intervals in the Cantor series expansion associated to the sequence $(M_n)$. Our random measures will be, essentially, projections via this Cantor series expansion of measures defined on the tree $\Sigma^*$.

Let $(X_\jjj)_{\jjj\in \Sigma^*}$ be a family of independent random variables satisfying the following properties for each $\jjj\in\Sigma_n$:
\begin{enumerate}
\renewcommand{\labelenumi}{(X\arabic{enumi})}
\renewcommand{\theenumi}{X\arabic{enumi}}
\item \label{it:size} Almost surely, $X_\jjj$ is a subset of $[M_{n+1}]$ with $|X_\jjj|=L_{n+1}$.
\item \label{it:martingale} For each $a\in [M_{n+1}]$, $\PP(a\in X_\jjj)=L_{n+1}/M_{n+1}$.
\end{enumerate}
We underline that we do \emph{not} assume that $X_\jjj$ is chosen uniformly among all subsets of $[M_{n+1}]$ of size $L_{n+1}$ (as in \cite[Section 6]{LabaPramanik09}); this additional flexibility in choosing $X_\jjj$ will be key in avoiding progressions.

These random variables yield a random subtree of $\Sigma^*$ in the natural way:
\[
\mathcal{J}_{n} = \{  (j_1\ldots j_n): j_{i+1}\in X_{j_1\ldots j_i} \text{ for all } j=0,\ldots,n-1 \}.
\]
From \eqref{it:size} and \eqref{it:martingale}, we have
\[
|\mathcal{J}_n| =L_1\cdots L_n \quad\text{almost surely},
\]
and
\begin{equation} \label{eq:symbolic-martingale}
\mathbb{P}((\jjj a)\in\mathcal{J}_{n+1}) = \mathbf{1}(\jjj\in \mathcal{J}_n) \frac{L_{n+1}}{M_{n+1}}\quad\text{for all } a\in [M_{n+1}].
\end{equation}
(We denote by $\mathbf{1}(\cdot)$ the indicator function of both events and sets.) Further, let us define the functions
\begin{equation} \label{eq:defmu}
\mu_n = \frac{M_1\cdots M_n}{L_1\cdots L_n} \, \sum_{\jjj\in\cJ_n} \1(I_\jjj).
\end{equation}
By the Caratheodory Extension Theorem, there exists a measure $\mu$ on $[0,1]$ (which we may identify with $\R/\Z$) such that
\[
\mu(I_\jjj) = \mu_n(I_\jjj) =  \frac{1}{L_1\cdots L_n} \quad\text{for all }\jjj\in \cJ_n.
\]
In particular, $\mu_n\to \mu$ weakly.

\begin{thm} \label{thm:Salem}
Suppose that the sequence $(M_n)$ satisfies
\begin{equation} \label{eq:condition-on-M}
\lim_{n\to\infty} \frac{\log M_{n+1}}{\log(M_1\cdots M_n)} = 0.
\end{equation}
Fix $\sigma>0$ such that
\begin{equation} \label{eq:condition-on-sigma}
\sigma < \liminf_{n\to\infty} \frac{\log(L_1\cdots L_n)}{\log(M_1\cdots M_n)}.
\end{equation}
Then almost surely there is $C>0$ such that
\[
|\wh{\mu}(k)| \le C\, |k|^{-\sigma/2} \quad\text{for all }k\in\Z\setminus\{0\}.
\]
\end{thm}

The proof of this theorem is a small variant of (and in some respects simpler than) \cite[Theorem 14.1]{ShmerkinSuomala15}, and is also closely related to the results in \cite[Section 6]{LabaPramanik09}. For completeness we give the full proof.

We recall Hoeffding's inequality \cite{Hoeffding63}:
\begin{lemma} \label{lem:Hoeffding}
Let $\{ Y_i: i\in J\}$ be zero mean independent complex-valued random variables satisfying $|Y_i|\le R$. Then for all $\kappa>0$,
\begin{equation} \label{eq:Hoeffding}
\mathbb{P}\left(\left|\sum_{i\in J} Y_i\right|> \kappa\right) \le 4\exp\left(\frac{-\kappa^2}{4 R^2 |J|}\right).
\end{equation}
\end{lemma}
We note that Hoeffding's inequality is usually stated for real-valued random variables, but the extension to complex-valued ones is immediate by looking at the real and imaginary parts.

\begin{proof}[Proof of Theorem \ref{thm:Salem}]
For notational simplicity, let
\begin{align*}
P_n &= L_1\cdots L_n,\\
Q_n &= M_1\cdots M_n.
\end{align*}

We write $e(x)=\exp(-2\pi i x)$ and define a (complex) measure $\eta_k = e(k x) dx$. Now
\[
\widehat{\mu}_{n+1}(k)-\widehat{\mu}_n(k)=\sum_{\jjj\in\Sigma_n}Y_\jjj = \sum_{\jjj\in\cJ_n} Y_\jjj,
\]
where
\[
Y_\jjj=\int_{I_\jjj}\mu_{n+1}-\mu_n\,d\eta_k.
\]
Note that, as a consequence of \eqref{eq:symbolic-martingale} and the definitions, $\EE(\mu_{n+1}(x)|\cJ_n)=\mu_n(x)$ for all $x$, so  $\EE(Y_\jjj|\cJ_n)=0$. Also, since $|\eta_k|\le 1$ and $|I_\jjj|=(M_1\cdots M_n)^{-1}$, from \eqref{eq:defmu} we get
\[
|Y_\jjj| \le \frac{2 M_{n+1}}{L_1\cdots L_n L_{n+1}}.
\]
Applying Hoeffding's inequality \eqref{eq:Hoeffding}, a short calculation yields
\[
\PP\left(|\widehat{\mu}_{n+1}(k)-\widehat{\mu}_n(k)| > Q_{n+1}^{-\sigma/2} \,|\, \cJ_n\right) \le 4 \exp\left(-\frac{L_{n+1}^2}{16 M_{n+1}^{2+\sigma}}   P_n   Q_{n}^{-\sigma}  \right).
\]
Since this bound holds uniformly, it also holds without conditioning on $\cJ_n$. Let
\begin{equation} \label{eq:Salem-large-dev}
\e_n = \PP\left(|\widehat{\mu}_{n+1}(k)-\widehat{\mu}_n(k)| > Q_{n+1}^{-\sigma/2} \text{ for some } |k|< Q_{n+1}\right).
\end{equation}
Recalling \eqref{eq:condition-on-M} and \eqref{eq:condition-on-sigma}, we see that if $n$ is large enough, then
\[
\e_n \le 8 Q_{n+1}\exp(-Q_n^\eta),
\]
for some $\eta=\eta(\sigma)>0$ (this is the point where the assumption on $\sigma$ gets used). Using  \eqref{eq:condition-on-M} again, together with $Q_n\ge 2^{-n}$, this shows that $(\e_n)$ is summable.

Let $\jjj\in\Sigma_{n+1}$. Since the endpoints of $I_\jjj$ are of the form $m/Q_{n+1}$, $m\in\Z$, it follows that for $|k|<Q_{n+1},0\neq\ell\in\mathbb{Z}$,
\[
\widehat{\mathbf{1}}_{I_\jjj}(k+Q_{n+1}\ell) = \frac{k}{k+ Q_{n+1}\ell} \,\widehat{\mathbf{1}}_{I_\jjj}(k)\,.
\]
Since $\widehat{\mu}_{n+1}$ is a linear combination of the functions $\widehat{\mathbf{1}}_{I_\jjj}$, the same relation holds between $\widehat{\mu}_{n+1}(k+Q_{n+1}\ell)$ and $\widehat{\mu}_{n+1}(k)$, and likewise for $\widehat{\mu}_{n}$. Hence
\[
\widehat{\mu}_{n+1}(k+Q_{n+1}\ell) - \widehat{\mu}_{n}(k+Q_{n+1}\ell) = \frac{k}{k+ Q_{n+1}\ell} \left( \wh{\mu}_{n+1}(k)-\wh{\mu}_n(k) \right).
\]

Note that we can write any frequency $k'$ with $|k'|\ge Q_{n+1}$ as $k'=k+Q_{n+1}\ell$, where $|k|<Q_{n+1}$, and $\ell\neq 0$. For such $k'$, we have
\[
|\widehat{\mu}_{n+1}(k')-\widehat{\mu}_n(k')| \le \frac{|k|}{|k'|}|\widehat{\mu}_{n+1}(k)-\widehat{\mu}_n(k)|
< \frac{Q_{n+1}}{|k'|}|\widehat{\mu}_{n+1}(k)-\widehat{\mu}_n(k)|.
\]
Let $E_n$ be the event
\[
\left|\widehat{\mu}_{n+1}(k)-\widehat{\mu}_n(k)\right|\le
\min\left(1,\frac{Q_{n+1}}{|k|} \right) Q_{n+1}^{-\sigma/2}\text{ for all }k\in\Z\,.
\]
We have seen that $\PP(E_n)>1-\e_n$, where $\e_n$ is defined in \eqref{eq:Salem-large-dev}. Since $\e_n$ was summable, by the Borel-Cantelli Lemma a.s. there is $n_0$ such that $E_n$ holds for all $n\ge n_0$.

Suppose $k$ is a frequency with $|k|\ge Q_{n_0}$. Choosing $n_1\in\N$ such that $Q_{n_1}\le|k|<Q_{n_1+1}$, and telescoping, for any $m>n_1$ we have
\begin{equation}\label{eq:telescope}
\begin{split}
\left|\widehat{\mu}_m(k)-\widehat{\mu}_{n_0}(k)\right|&\le  \sum_{n_0\le n< n_1} |\wh{\mu}_{n+1}(k)-\wh{\mu}_n(k)| + \sum_{n_1\le n<m} |\wh{\mu}_{n+1}(k)-\wh{\mu}_n(k)| \\
&\le\sum_{n_0\le n< n_1} Q_{n+1}^{1-\sigma/2}|k|^{-1} +\sum_{n_1\le n< m}Q_{n+1}^{-\sigma/2}\\
&\le 2 Q_{n_1}^{1-\sigma/2} |k|^{-1} + 2 Q_{n_1+1}^{-\sigma/2}\\
&\le 4 Q_{n_1}^{-\sigma/2},
\end{split}
\end{equation}
where in the third line we used that $Q_{n+1}\ge 2 Q_n$ for all $n$. By \eqref{eq:condition-on-M}, given $\delta>0$ it holds that $|k|\le Q_{n_1}^{1+\delta}$ for all sufficiently large $k$, whence
\[
|\widehat{\mu}_m(k)| \le |\widehat{\mu}_{n_0}(k)| + 4|k|^{-(1+\delta)\sigma/2}
\]
provided $k$ is sufficiently large. Noting that $|\widehat{\mu}_{n_0}(k)|\le C'(n_0)|k|^{-1}$ and letting $m\to\infty$ finishes the proof.
\end{proof}

\subsection{Behrend's example for unions of intervals}

In the following simple lemma, we adapt Behrend's classical example to a setting more suitable for our needs. Given $m\in\N, j\in\Z/m\Z$, let $\mathcal{I}_{m,j}=[j/m,(j+1)/m)\subset\R/\Z$. Moreover, given $X\subset\Z/m\Z$, we write $\mathcal{I}(X) = \cup_{j\in X} \mathcal{I}_{m,j}\subset \R/\Z$ for the union of the $m$-adic intervals corresponding to $X$.
\begin{lemma} \label{lem:Behrend}
For any $\e>0$, there exist an integer $m_0$, such that for all $m\ge m_0$ there is a set $X\subset\Z/m\Z$ satisfying:
\begin{enumerate}
\item[(i)] $|X|> m^{1-\e}$.
\item[(ii)] If $\{x,y,z\}\subset \mathcal{I}(X+\ell)$ is a $3$-$\AP$ for some $\ell\in \Z/m\Z$, then there is $j\in X$ such that $\{ x,y,z\}\subset \mathcal{I}_{m,j+\ell}$.
\end{enumerate}
\end{lemma}
\begin{proof}
Let $m'=\lfloor m/5\rfloor$. Behrend \cite{Behrend46} (see also \cite{Elkin11} for a recent improvement) constructed a set $X'\subset \{1,\ldots,m'\}$ with no $3$-$\AP$s and
\[
|X'|\ge m'\exp(-c\sqrt{\log m'}),
\]
where $c>0$ is an absolute constant.

Let $X=2 X'= \{2x:x\in X'\}$ embedded in $\Z/m\Z$. It is clear that (i) holds if $m$ is large enough depending on $\e$. For the second part, we may assume $\ell=0$ as translating the set does not affect the claim. Suppose $x<y<z$ forms a $3$-$\AP$ in $\mathcal{I}(X)$ such that $x\in \mathcal{I}_{m,2j}, z\in \mathcal{I}_{m,2k}$ with $j\neq k$. Write $x=(2j+\delta)/m$, $z=(2k+\delta')/m$, with $0\le \delta,\delta'< 1$. Then
\[
y = \frac{x+z}{2} = (j+k) + \frac{\delta+\delta'}{2}.
\]
Thus $y\in \mathcal{I}_{m,j+k}\cap \mathcal{I}(X)$, but this implies that $\{ 2j,j+k,2k\}$ is a non-trivial $3$-$\AP$ in $X$, which is impossible by construction.
\end{proof}

\subsection{Proof of main results}

By iterating the set transformation $[0,1]\mapsto \mathcal{I}(X)$, one can construct a self-similar set of dimension $>1-\e$ with no progressions, but such self-similar sets are as far as possible from being Salem (they do not support any measure whose Fourier transform tends to $0$ at infinity). The main idea of this paper is to modify the construction of this self-similar set by translating the set $X$ by a uniformly chosen $\ell\in \Z/m\Z$ every time an interval is split, with all the choices independent. This produces a measure that falls into the framework of Theorem \ref{thm:Salem} but still has an $\AP$-free support.

\begin{proof}[Proof of Theorem \ref{mainthm:dimless1}]
Fix $t\in (0,1)$, and apply Lemma \ref{lem:Behrend} with $\e=1-t$ to obtain $M\in\N$ and a set $X\subset\Z/M\Z$ satisfying the conditions of the lemma, so that in particular $M^t < |X| < M$.

Let $L_n$ be the following sequence: $L_1=|X|$. Once $L_n$ is defined, set $L_{n+1}=|X|$ if $L_1\cdots L_n < M^{(n+1)t}$ and $L_{n+1}=1$ otherwise. Then
\begin{equation} \label{eq:symbolic-AR}
M^{nt} \le L_1\cdots L_n < |X| M^{nt}
\end{equation}
for all $n$. We also set $M_n\equiv M$.

Let $\{\ell_\jjj:\jjj\in\Sigma^*\}$ be a sequence of independent random variables, distributed uniformly in $[M]$, and set
\[
X_\jjj = \left\{    \begin{array}{ll}
              X+\ell_\jjj \bmod M & \text{ if } L_{n+1}=|X| \\
              \{ \ell_\jjj\} & \text{ if } L_{n+1}=1
            \end{array}
\right..
\]
It is clear that the $X_\jjj$ are independent (since the $\ell_\jjj$ are), and satisfy properties \eqref{it:size}-\eqref{it:martingale}.

Let $\mu_n$ and $\mu$ be defined from this data as in Section \ref{subsec:Salem}. We will check that $\mu$ satisfies the required properties.

Note that \eqref{eq:condition-on-M} holds trivially in this case. Also,
\[
\lim_{n\to\infty} \frac{\log(L_1\cdots L_n)}{\log(M_1\cdots M_n)} = t
\]
by \eqref{eq:symbolic-AR}. Theorem \ref{thm:Salem} then asserts that the Salem condition $(\mathsf{i})$ holds (more precisely, take a sequence $\sigma_j\uparrow t$, and apply Theorem  \ref{thm:Salem} to each $\sigma_j$).

Given $r\in (0,1]$, pick $n$ such that $M^{-n} \le r < M^{1-n}$. Since $(x-r,x+r)$ hits at most $(2M+1)$ intervals $I_\jjj, \jjj\in\Sigma_n$, we see from \eqref{eq:symbolic-AR}  that
\[
 \mu(x-r,x+r) \le (2M+1)P_n^{-1} \le (2M+1) r^t.
\]
On the other hand, if $x\in\supp(\mu)$, then $(x-r,x+r)$ contains one interval $I_\jjj$ with $\jjj\in\Sigma_n$, so
\[
 \mu(x-r,x+r) \ge (L_1\cdots L_n)^{-1} \ge \frac{1}{M^t |X|} \,r^t.
\]
This shows that the Ahlfors-regularity condition \textsf{(ii)} holds.

It remains to see that $\supp(\mu)$ contains no $3$-$\AP$s. Firstly, we note that almost surely $\supp(\mu)$ does not contain any endpoint of an interval $I_\jjj$, since there are only countably many such endpoints, and each of them has probability zero of surviving. Now if $x<y<z$ is a $3$-$\AP$ in $\supp(\mu)$ where none of the points are endpoints, we can pick a minimal interval  $I_\jjj, \jjj\in\Sigma^*$ such that $\{x,y,z\}\subset I_\jjj$. By minimality, $L_n=|X|$ where $n$ is the length of $\jjj$. By construction,
\[
\bigcup\{ I_{\jjj a}:a\in X_\jjj\}
\]
is an affine copy of $\mathcal{I}(X+\ell_\jjj \bmod M)$. Lemma \ref{lem:Behrend} then says that $\{x,y,z\}\subset I_{\jjj a}$ for some $a$, but this contradicts the minimality of $I_\jjj$.

This contradiction establishes \textsf{(iii)} and finishes the proof.
\end{proof}

\begin{proof}[Proof of Theorem \ref{mainthm:dimeq1}]
The proof is very similar to that of Theorem \ref{mainthm:dimless1}. For each $n\in\N$, let $X^{(n)}\subset \Z/n\Z$ have maximal size satisfying the second part of Lemma \ref{lem:Behrend}. Then, by the first part of the lemma, $\log|X^{(n)}|/\log n\to 1$ as $n\to\infty$.

We define the sequences $L_n=|X^{(n)}|$, $M_1=2$, $M_n=n$ if $n\ge 2$. Then
\begin{equation} \label{eq:measure-full-dim}
\lim_{n\to\infty} \frac{\log(L_1\cdots L_n)}{\log(M_1\cdots M_n)} = 1,
\end{equation}
and \eqref{eq:condition-on-M} holds.

 Let $\{\ell_\jjj:\jjj\in\Sigma^*\}$ be a sequence of independent random variables, with $\ell_\jjj$ distributed uniformly in $[n+1]$ for $\jjj\in\Sigma_n$. We set
\[
X_\jjj = X^{(n+1)} + \ell_\jjj \bmod (n+1)
\]
for $\jjj\in\Sigma_n$, and note that the $X_\jjj$ are independent and satisfy \eqref{it:size} and \eqref{it:martingale}.

Let $\mu_n$ and $\mu$ be as in Section \ref{subsec:Salem}. We have checked that the hypotheses of Theorem \ref{thm:Salem} hold and, by \eqref{eq:measure-full-dim}, the theorem can be applied to a sequence $\sigma_j \uparrow 1$. This establishes the first claim.

For the second claim, note that (by construction) given $\e>0$, there is  $C_\e>0$ such that $|X^{(n)}|\ge C_\e n^{1-\e}$ for all $n$. Now given $r\in (0,1]$, pick $n$ with
\[
\frac{1}{(n+1)!} \le r < \frac{1}{n!}.
\]
Then $(x-r,x+r)$ meets at most $2$ intervals $I_\jjj, \jjj\in \Sigma_n$, so
\begin{align*}
\mu(x-r,x+r)&\le 2 (|X^{(1)}|\cdots |X^{(n)}|)^{-1} \\
&\le 2 C_\e^{-n} (n!)^{\e-1} \le ((n+1)!)^{2\e-1} \le r^{1-2\e},
\end{align*}
provided $n$ is large enough (so that $r$ must be small enough). This shows that the Frostman condition $(\mathsf{ii})$ holds.

The fact that $\supp(\mu)$ has no progressions follows exactly as in the proof of Theorem \ref{mainthm:dimless1}: if $x<y<z$ was such a progression, and $I_\jjj$ was the smallest interval in the construction containing $\{x,y,z\}$, it would follow from Lemma \ref{lem:Behrend} that $\{x,y,z\}\in I_{\jjj a}$ for some $a$, contradicting minimality.
\end{proof}


\end{document}